\documentclass[11pt]{article}

\usepackage[margin=1in]{geometry}
\usepackage{amsmath,amssymb,amsthm,mathtools}
\usepackage{enumitem}
\usepackage{microtype}
\usepackage[hidelinks]{hyperref}
\usepackage{xcolor}

\newtheorem{theorem}{Theorem}[section]
\newtheorem{lemma}[theorem]{Lemma}

\newtheorem{corollary}[theorem]{Corollary}
\theoremstyle{definition}
\newtheorem{definition}[theorem]{Definition}
\theoremstyle{remark}
\newtheorem{remark}[theorem]{Remark}

\newcommand{\blfootnote}[1]{%
	\begingroup
	\renewcommand\thefootnote{}\footnote{#1}%
	\addtocounter{footnote}{-1}%
	\endgroup
}

\newcommand{\Ap}{\operatorname{Ap}}
\newcommand{\Z}{\mathbb Z}
\newcommand{\Nzero}{\mathbb Z_{\ge 0}}
\newcommand{\brp}[3]{p_{\le #1}^{(\le #2)}(#3)}
\newcommand{\qbinom}[2]{\genfrac{[}{]}{0pt}{}{#1}{#2}_{q}}

\title{$p$-NUMERICAL SEMIGROUP OF THE SEQUENCE OF CONSECUTIVE ODD INTEGERS}
\usepackage{authblk}

\author[1,2]{Takao Komatsu}
\author[2]{Sungjin Hyun}
\author[2,*]{Kyunghwan Song} 

\affil[1]{Department of Mathematics, Institute of Science Tokyo, 152-8551 Japan}
\affil[2]{Department of Mathematics, Jeju National University, 102 Jejudaehakro Jeju, 63243, Republic of Korea}
\date{}

\begin{document}

\maketitle
\blfootnote{* Corresponding author: khsong@jejunu.ac.kr}

\begin{abstract}
We prove the $p$-Frobenius problems proposed as 
Conjectures 7.1 and 7.5 in \cite{KP2025} for two families of consecutive odd integers. 
For integers $r,L,n\ge0$, the bounded restricted partition
function $\brp{r}{L}{n}$ counts partitions of $n$ into at most $r$ parts, each at most $L$.
Thus the bounded restricted partition functions $\brp{3}{a}{s}$ and $\brp{3}{a+1}{s}$ play central roles in the proofs. Their generating functions are Gaussian polynomials, whose symmetry and unimodality provide a common tool for treating both families.
\end{abstract}

\section{Definitions and main results}

Let $A = \{a_1,\ldots,a_k\}$ be a finite set of positive integers with $\gcd(A) = 1$.  For $n\in\Nzero$, define
$$
r_A(n) = \#\left\{
(x_1,\ldots,x_k)\in\Nzero^k:
x_1a_1+\cdots+x_ka_k=n
\right\}.
$$
Thus $r_A(n)$ counts unordered representations, or equivalently
factorizations by multiplicity vectors.

For $p\in\Nzero$, put
$$
S_p(A)=\{n\in\Nzero:r_A(n)>p\}.
$$
The $p$-Frobenius number is
$$
g_p(A)=\max\bigl(\Nzero\setminus S_p(A)\bigr).
$$
If $b = \min A$, the $p$-Ap\'ery set with respect to $b$ is
$$
\Ap_p(A;b) = \{m_0^{(p)},m_1^{(p)},\ldots,m_{b-1}^{(p)}\},
$$
where $m_i^{(p)}$ is the least member of $S_p(A)$ congruent to
$i\pmod b$. 

The paper \cite{KP2025} shows explicit formulas of the $p$-Frobenius number and related values when the elements in $A$ consist of the sequence of consecutive integers $a,a + 1,\dots,2 a-1$. 
This result seems to be easy to generalize or find variations, but in fact, even in the case of the odd number, it is hard to get the results, and only conjectures are given in \cite{KP2025}. In this paper, we give the $p$-Frobenius numbers by verifying Conjecture 7.1 in the case of consecutive odd numbers $2a+1, 2 a+3,\dots,4 a+1$, and Conjecture 7.5 in the case of consecutive odd numbers $2a+1, 2 a+3,\dots,4 a+3$.

\subsection{Restricted and bounded restricted partitions}

\begin{definition}[Bounded restricted partition function]
\label{def:bounded-partition}
For integers $r,L,n\ge0$, define
$$
\brp{r}{L}{n}
=
\#\left\{
(\lambda_1,\ldots,\lambda_r)\in\Nzero^r:
L\ge\lambda_1\ge\cdots\ge\lambda_r\ge0,\quad
\lambda_1+\cdots+\lambda_r=n
\right\}.
$$
Equivalently, $\brp{r}{L}{n}$ is the number of partitions of $n$
into at most $r$ parts, with every part at most $L$.  We set
$\brp{r}{L}{n}=0$ for $n<0$.
\end{definition}

When $L\ge n$, the upper bound on the parts is automatic.  We therefore
write $p_{\le r}(n):=\brp{r}{n}{n}$ for the number of partitions of $n$ into at most $r$ parts, and set
$p_{\le r}(n)=0$ for $n<0$.  By conjugation, $p_{\le r}(n)$ also
counts partitions whose parts are at most $r$.  These are classical
restricted partition functions; see \cite[Chapter 1]{Andrews}.  For
$r=3,4,5,6$, the sequences are listed in \cite{Sloane2026} as A001399,
A001400, A001401, and A001402, respectively.

In particular, the two bounded quantities needed later are $\brp{3}{a}{s}$ and $\brp{3}{a+1}{s}.$
By Definition~\ref{def:bounded-partition}, these count the nonincreasing
triples with entries bounded by $a$ and $a+1$, respectively.  

For simplicity, set $M_m:=p_{\le4}(m)$ for $m\ge0.$ The sequence $(M_m)_{m\ge0}$ appears in \cite[A001400]{Sloane2026}.

\subsection{The two results}

The first family is
$$
A_a^-=\{2a+1,2a+3,\ldots,4a+1\}.
$$
It is the family appearing in Conjecture 7.5.

\begin{theorem}
\label{thm:minus}
Let $a\ge4$, and let
$$
1\le p\le\left\lfloor\frac{a+1}{2}\right\rfloor.
$$
Let $m\ge1$ be determined by
$$
M_m\le p<M_{m+1}.
$$
Then
$$
g_p(A_a^-)=8a+2m+4.
$$
\end{theorem}

The second family is
$$
A_a^+=\{2a+1,2a+3,\ldots,4a+3\}.
$$
The following is the corrected positive-$p$ statement for the family
from Conjecture 7.1.

\begin{theorem}
\label{thm:plus}
For every $a\ge1$,
$$
g_1(A_a^+)=6a+5.
$$
Moreover, let $a\ge3$ and
$$
2\le p\le\left\lfloor\frac{a+1}{2}\right\rfloor.
$$
If $j\ge2$ is determined by
$$
M_j\le p<M_{j+1},
$$
then
$$
g_p(A_a^+)=8a+2j+4.
$$
\end{theorem}

\begin{remark}
The $p=0$ assertions are the classical Frobenius formulas already
recorded with the original conjectures.  The proofs below address the
positive-$p$ assertions, where the restricted-partition structure is
essential. The hypothesis $a\geq 4$ in Theorem 1.2 is imposed in order to obtain
	a uniform statement and proof. Among the omitted small cases, the
	formula remains valid for $(a,p)=(2,1)$ and $(3,1)$, whereas it
	fails for
	$$
	(a,p)=(1,1)
	\quad\text{and}\quad
	(a,p)=(3,2).
	$$
	Indeed,
	$$
	g_1(A_1^-)=22\ne14,
	\qquad
	g_2(A_3^-)=37\ne32.
	$$
\end{remark}

\section{Common partition and Ap\'ery lemmas}

\begin{lemma}
\label{lem:gaussian}
For $r,L\ge0$,
$$
\sum_{n=0}^{rL}\brp{r}{L}{n}q^n
=
\qbinom{L+r}{r}
:=
\prod_{i=1}^{r}\frac{1-q^{L+i}}{1-q^i}.
$$
The expression on the right is the Gaussian binomial coefficient
(or $q$-binomial coefficient), with the empty product interpreted as
$1$ when $r=0$.
Consequently, the finite sequence $\brp{r}{L}{0},\allowbreak \brp{r}{L}{1},\allowbreak \ldots,\allowbreak \brp{r}{L}{rL}$ is symmetric and unimodal.
\end{lemma}

\begin{proof}
The Ferrers diagram of a partition counted by $\brp{r}{L}{n}$
fits inside an $r\times L$ rectangle.  The standard generating
function for such diagrams is the Gaussian polynomial
$\qbinom{L+r}{r}$.  The symmetry follows either by taking complements
inside the rectangle or from the reciprocity of the Gaussian polynomial. 
Its unimodality follows from O'Hara's constructive proof of the
	unimodality of Gaussian coefficients; see
	\cite[Theorem 3.5]{OHara}.
\end{proof}

\begin{corollary}
\label{cor:three-part-bound}
For $L\ge0$ and $L\le s\le2L$,
$$
\brp{3}{L}{s}\ge\brp{3}{L}{L}=p_{\le3}(L).
$$
\end{corollary}

\begin{proof}
By Lemma \ref{lem:gaussian}, the coefficient sequence is symmetric about
$3L/2$ and unimodal.  Hence its minimum on $[L,2L]$ occurs at an
endpoint, and symmetry gives $\brp{3}{L}{2L}=\brp{3}{L}{L}$.
Finally, every partition of $L$ into at most three parts automatically
has largest part at most $L$, so
$\brp{3}{L}{L}=p_{\le3}(L)$.
\end{proof}

\begin{lemma}
\label{lem:index-bound}
If 
$$
M_t\le\left\lfloor\frac{a+1}{2}\right\rfloor,
$$ 
then $t\le a$.
\end{lemma}

\begin{proof}
Since
$$
M_t=p_{\le4}(t)\ge p_{\le2}(t)
=\left\lfloor\frac t2\right\rfloor+1,
$$
the assumption implies
$$
a\ge2M_t-1
\ge2\left\lfloor\frac t2\right\rfloor+1
\ge t.
$$
\end{proof}

\begin{lemma}
\label{lem:closure}
For every $p\ge0$,
$$
S_p(A)+A\subseteq S_p(A).
$$
In particular, if $b\in A$, then
$$
S_p(A)+b\subseteq S_p(A).
$$
\end{lemma}

\begin{proof}
Fix $c\in A$.  In multiplicity-vector notation, adding $c$ to a
factorization of $n$ adds $1$ to the coordinate belonging to $c$.
This gives an injection from the factorizations of $n$ to those of
$n+c$.  Thus 
$$
r_A(n+c)\ge r_A(n).
$$ 
If $n\in S_p(A)$, then $r_A(n)>p$, and hence $n+c\in S_p(A)$.
\end{proof}

\begin{lemma}
\label{lem:apery-formula}
If $b\in A$, then
$$
g_p(A)=\max\Ap_p(A;b)-b.
$$
\end{lemma}

\begin{proof}
This identity is established in
\cite[Corollary 1, Eq.(4)]{Komatsu2024}.  We include the following
short alternative proof for completeness.

Write $W=\max\Ap_p(A;b)$.  By minimality of the Ap\'ery element in
its residue class, $W-b\notin S_p(A)$, so $g_p(A)\ge W-b$.
If $n>W-b$, then the least element $m_i^{(p)}$ in the residue class
of $n$ satisfies $m_i^{(p)}\le W<n+b$.  Hence
$n=m_i^{(p)}+kb$ for some $k\ge0$.  Lemma \ref{lem:closure} gives
$n\in S_p(A)$.  Therefore $g_p(A)=W-b$.
\end{proof}

\section{Proof for the family \texorpdfstring{$A_a^-$}{A-a-minus}}

Throughout this section, let
$$
b=2a+1,\qquad
A_a^-=\{b+2t:0\le t\le a\},
$$
and write $r_-(n)=r_{A_a^-}(n)$.

\begin{lemma}
\label{lem:minus-four}
For $0\le s\le a$,
$$
r_-(4b+2s)=p_{\le4}(s).
$$
\end{lemma}

\begin{proof}
Suppose that a representation has $q$ summands:
$$
4b+2s=\sum_{i=1}^{q}(b+2t_i),
\qquad 0\le t_i\le a.
$$
Then
$$
2\sum_{i=1}^{q}t_i=(4-q)b+2s.
$$
Since $b$ is odd, $q$ is even.  If $q=2$, then
$$
t_1+t_2=s+b>2a,
$$
which is impossible.  If $q\ge6$, then
$$
\sum_{i=1}^{q}t_i
=s-\frac{q-4}{2}b<0,
$$
which is also impossible.  Thus $q=4$, and cancellation of $4b$
gives
$$
t_1+t_2+t_3+t_4=s.
$$
Because $s\le a$, the bounds $t_i\le a$ are automatic.  After the
indices are arranged in nonincreasing order, the representations are
in bijection with the partitions of $s$ into at most four parts.
\end{proof}

\begin{lemma}
\label{lem:minus-five}
For $0\le s\le a$,
$$
r_-(5b+2s)
=p_{\le5}(s)+p_{\le3}(a-s-1).
$$
Moreover, for $a+1\le s\le2a$,
$$
r_-(3b+2s)=\brp{3}{a}{s}.
$$
\end{lemma}

\begin{proof}
First consider $5b+2s$.  If a representation has $q$ summands, then
$$
2\sum_{i=1}^{q}t_i=(5-q)b+2s.
$$
Parity forces $q$ to be odd.  The cases $q=1$ and $q\ge7$ are
impossible, so only $q=5$ and $q=3$ occur.

For $q=5$, cancellation gives
$$
t_1+\cdots+t_5=s.
$$
Since $s\le a$, these representations are counted by
$p_{\le5}(s)$.  For $q=3$,
$$
t_1+t_2+t_3=s+b=s+2a+1.
$$
For simplicity, set $u_i=a-t_i$.  Then
$$
u_1+u_2+u_3=a-s-1.
$$
When the right-hand side is nonnegative, its size is at most $a$, so
the upper bounds $u_i\le a$ are automatic.  These representations are
therefore counted by $p_{\le3}(a-s-1)$; our convention makes this term
zero when $a-s-1<0$.

Now let $a+1\le s\le2a$ and consider $3b+2s$.  A representation with
$q$ summands satisfies
$$
2\sum_{i=1}^{q}t_i=(3-q)b+2s.
$$
Again $q$ is odd.  The case $q=1$ would require $t_1=s+b>a$, and
if $q\ge5$, then
$$
\sum_{i=1}^{q}t_i
\le s-b<0
$$
because $s\le2a=b-1$.  Hence $q=3$, and the representation condition
is precisely
$$
a\ge t_1\ge t_2\ge t_3\ge0,
\qquad t_1+t_2+t_3=s.
$$
By Definition~\ref{def:bounded-partition}, the number of such triples is
$\brp{3}{a}{s}$.
\end{proof}

\begin{lemma}
\label{lem:minus-lower}
Let $a\ge4$.  Then
$$
p_{\le3}(a)>
\left\lfloor\frac{a+1}{2}\right\rfloor
$$
and, for every $0\le s\le a$,
$$
p_{\le5}(s)+p_{\le3}(a-s-1)>
\left\lfloor\frac{a+1}{2}\right\rfloor.
$$
\end{lemma}

\begin{proof}
We have
$$
p_{\le3}(a)\ge p_{\le2}(a)
=\left\lfloor\frac a2\right\rfloor+1.
$$
This is already strict relative to
$\lfloor(a+1)/2\rfloor$ when $a$ is even.  If $a$ is odd, then
$a\ge5$, and the partition
$$
a=(a-2)+1+1
$$
is counted by $p_{\le3}(a)$ but not by $p_{\le2}(a)$.

For the second assertion, first let $0\le s\le a-1$.  The elementary
bounds
$$
p_{\le5}(s)\ge
\left\lfloor\frac s2\right\rfloor+1,\qquad
p_{\le3}(a-s-1)\ge
\left\lfloor\frac{a-s-1}{2}\right\rfloor+1
$$
give
$$
p_{\le5}(s)+p_{\le3}(a-s-1)
\ge
\left\lfloor\frac s2\right\rfloor+
\left\lfloor\frac{a-s-1}{2}\right\rfloor+2.
$$
If $a$ is even, the right-hand side equals
$a/2+1$, which is greater than
$\lfloor(a+1)/2\rfloor$.

Let $a=2r+1$ be odd.  If $s$ is even, the displayed lower bound is
$r+2>r+1$.  If $s$ is odd, it initially gives $r+1$.
For $s\ge3$, the partition $s=(s-2)+1+1$ shows that
$p_{\le5}(s)>p_{\le2}(s)$.  If $s=1$, then $a-2\ge3$, and
$(a-2)=(a-4)+1+1$ shows that
$p_{\le3}(a-2)>p_{\le2}(a-2)$.  Thus the inequality is strict in
every case.

Finally, if $s=a$, then
$$
p_{\le5}(a)\ge p_{\le3}(a)>
\left\lfloor\frac{a+1}{2}\right\rfloor,
$$
and $p_{\le3}(-1)=0$.
\end{proof}

\begin{proof}[Proof of Theorem~\ref{thm:minus}]
By the choice of $m$,
$$
M_m\le p\le\left\lfloor\frac{a+1}{2}\right\rfloor.
$$
Lemma~\ref{lem:index-bound} therefore gives $m\le a$.  Applying
Lemma~\ref{lem:minus-four} with $s=m$, we obtain
$$
r_-(4b+2m)=p_{\le4}(m)=M_m\le p.
$$
Thus
\begin{equation}
4b+2m\notin S_p(A_a^-). \label{eq:minus-notin}
\end{equation}

We next produce, in every residue class modulo $b$, an element of
$S_p(A_a^-)$ not exceeding $5b+2m$.  Since $b$ is odd,
multiplication by $2$ permutes $\Z/b\Z$.  Hence every residue class
has a unique form
$$
2s\pmod b,\qquad 0\le s\le2a.
$$
The following three cases partition this entire range.

\medskip
\noindent\textbf{Case 1: $m+1\le s\le a$.}
Lemma~\ref{lem:minus-four} and the monotonicity of $p_{\le4}$ give
$$
r_-(4b+2s)
=p_{\le4}(s)
\ge p_{\le4}(m+1)
=M_{m+1}>p.
$$
Therefore $4b+2s\in S_p(A_a^-)$.  Moreover,
$$
4b+2s\le4b+2a<5b+2m.
$$

\medskip
\noindent\textbf{Case 2: $a+1\le s\le2a$.}
By the second formula in Lemma~\ref{lem:minus-five} and
Corollary~\ref{cor:three-part-bound},
$$
r_-(3b+2s)
=\brp{3}{a}{s}
\ge p_{\le3}(a).
$$
Lemma~\ref{lem:minus-lower} now yields
$$
r_-(3b+2s)>
\left\lfloor\frac{a+1}{2}\right\rfloor\ge p.
$$
Hence $3b+2s\in S_p(A_a^-)$, and
$$
3b+2s\le3b+4a=5b-2<5b+2m.
$$

\medskip
\noindent\textbf{Case 3: $0\le s\le m$.}
The first formula in Lemma~\ref{lem:minus-five}, followed by
Lemma~\ref{lem:minus-lower}, gives
$$
r_-(5b+2s)
=p_{\le5}(s)+p_{\le3}(a-s-1)
>
\left\lfloor\frac{a+1}{2}\right\rfloor
\ge p.
$$
Thus $5b+2s\in S_p(A_a^-)$, and
$$
5b+2s\le5b+2m.
$$

These cases cover all $b$ residue classes.  Consequently,
\begin{equation}
\max\Ap_p(A_a^-;b)\le5b+2m. \label{eq:minus-upper}
\end{equation}
Case~3 with $s=m$ also shows that $5b+2m\in S_p(A_a^-)$.

We claim that this is the least member of $S_p(A_a^-)$ in the residue
class $2m\pmod b$.  Indeed, since $m\le a$, we have $0\le2m<b$.
Any smaller nonnegative integer in this residue class has the form
$kb+2m$ with $0\leq k<5$.  If such an integer belonged to $S_p(A_a^-)$,
then Lemma~\ref{lem:closure}, applied $4-k$ times to the generator
$b$, would imply
$$
4b+2m=(kb+2m)+(4-k)b\in S_p(A_a^-),
$$
contradicting \eqref{eq:minus-notin}.  Hence
$$
m_{2m}^{(p)}=5b+2m.
$$
Together with \eqref{eq:minus-upper}, this proves
$$
\max\Ap_p(A_a^-;b)=5b+2m.
$$
Finally, Lemma~\ref{lem:apery-formula} gives
$$
g_p(A_a^-)
=(5b+2m)-b
=4b+2m
=8a+2m+4.
$$
\end{proof}

\section{Proof for the family \texorpdfstring{$A_a^+$}{A-a-plus}}

Throughout this section, let
$$
b=2a+1,\qquad
A_a^+=\{b+2t:0\le t\le a+1\},
$$
and write $r_+(n)=r_{A_a^+}(n)$.

\subsection{The case \texorpdfstring{$p=1$}{p=1}}

\begin{lemma}
\label{lem:plus-p1-count}
For every $a\ge1, r_+(3b+2)=1.$
\end{lemma}

\begin{proof}
Suppose that
$$
3b+2=\sum_{i=1}^{q}(b+2t_i),
\qquad 0\le t_i\le a+1.
$$
Then
$$
2\sum_{i=1}^{q}t_i=(3-q)b+2.
$$
Since $b$ is odd, $q$ is odd.  If $q=1$, then
$t_1=b+1=2a+2>a+1$.  If $q\ge5$, then
$$
\sum_{i=1}^{q}t_i
=1-\frac{q-3}{2}b<0.
$$
Thus $q=3$, and $t_1+t_2+t_3=1$.  The only unordered solution is
$(1,0,0)$.
\end{proof}

\begin{lemma}
\label{lem:plus-p1-cover}
Every residue class modulo $b$ contains an element of
$S_1(A_a^+)$ not exceeding $4b+2$.
\end{lemma}

\begin{proof}
Every residue class has a unique form $2s\pmod b$ with
$0\le s\le2a$.

For $s=0$, the two distinct representations
$$
4b=b+b+b+b=(b+2a)+(b+2a+2)
$$
show that $4b\in S_1(A_a^+)$.

For $s=1$,
$$
4b+2=(b+2)+b+b+b=(b+2a+2)+(b+2a+2),
$$
so $4b+2\in S_1(A_a^+)$.

If $2\le s\le a+1$, then
$$
3b+2s=(b+2s)+b+b
$$
and
$$
3b+2s=(b+2s-2)+(b+2)+b
$$
are distinct representations by elements of $A_a^+$.  Also
$$
3b+2s\le3b+2(a+1)=4b+1.
$$

Finally, if $a+2\le s\le2a$, then
$$
2b+2s
=(b+2(s-a-1))+(b+2(a+1))
$$
and
$$
2b+2s
=(b+2(s-a))+(b+2a)
$$
are distinct.  All four indices lie in $[0,a+1]$, and
$$
2b+2s\le2b+4a=4b-2.
$$
The four cases cover every residue class.
\end{proof}

\begin{proof}[Proof of Theorem~\ref{thm:plus} for $p=1$]
Lemma~\ref{lem:plus-p1-count} gives
$$
3b+2\notin S_1(A_a^+),
$$
whereas Lemma~\ref{lem:plus-p1-cover} gives
$$
4b+2\in S_1(A_a^+)
\quad\text{and}\quad
\max\Ap_1(A_a^+;b)\le4b+2.
$$
We show that $4b+2$ is the $1$-Ap\'ery element in the residue class
$2\pmod b$.  If a smaller number $kb+2$, with $0\leq k<4$, belonged to
$S_1(A_a^+)$, then Lemma~\ref{lem:closure} would give
$$
3b+2=(kb+2)+(3-k)b\in S_1(A_a^+),
$$
a contradiction.  Hence
$$
\max\Ap_1(A_a^+;b)=4b+2.
$$
By Lemma~\ref{lem:apery-formula},
$$
g_1(A_a^+)
=(4b+2)-b
=3b+2
=6a+5.
$$
\end{proof}

\subsection{The case \texorpdfstring{$p\ge2$}{p at least 2}}

\begin{lemma}
\label{lem:plus-four}
For $2\le s\le a+1, r_+(4b+2s)=p_{\le4}(s).$
\end{lemma}

\begin{proof}
For a representation with $q$ summands,
$$
2\sum_{i=1}^{q}t_i=(4-q)b+2s,
$$
so $q$ is even.  If $q=2$, then
$$
t_1+t_2=s+b\ge2a+3>2(a+1),
$$
which is impossible.  If $q\ge6$, then
$$
\sum_{i=1}^{q}t_i
\le s-b<0
$$
because $s\le a+1<b$.  Thus $q=4$, and
$$
t_1+t_2+t_3+t_4=s.
$$
The upper bounds $t_i\le a+1$ are automatic.  The resulting unordered
quadruples are counted by $p_{\le4}(s)$.
\end{proof}

\begin{lemma}
\label{lem:plus-layers}
One has
$$
r_+(5b)=p_{\le3}(a+2).
$$
For $1\le s\le a+1$,
$$
r_+(5b+2s)
=p_{\le5}(s)+p_{\le3}(a+2-s).
$$
For $a+2\le s\le2a$,
$$
r_+(3b+2s)=\brp{3}{a+1}{s}.
$$
\end{lemma}

\begin{proof}
Consider a representation of $5b+2s$.  As in the proof of
Lemma~\ref{lem:minus-five}, only $q=5$ and $q=3$ can occur.
For $q=5$,
$$
t_1+\cdots+t_5=s.
$$
For $q=3$,
$$
t_1+t_2+t_3=s+b.
$$
Putting $u_i=a+1-t_i$ gives
\begin{equation}
u_1+u_2+u_3=a+2-s. \label{eq:plus-complement}
\end{equation}
If $1\le s\le a+1$, the right-hand side lies in $[1,a+1]$, so its
parts automatically satisfy $u_i\le a+1$.  The $q=5$ and $q=3$
contributions are therefore $p_{\le5}(s)$ and
$p_{\le3}(a+2-s)$, respectively.

When $s=0$, the five-term contribution is the single partition of
$0$.  In the three-term contribution, \eqref{eq:plus-complement}
counts all partitions of $a+2$ into at most three parts except the
one-part partition $(a+2)$, which violates $u_1\le a+1$.  Therefore
$$
r_+(5b)
=1+\bigl(p_{\le3}(a+2)-1\bigr)
=p_{\le3}(a+2).
$$

Finally, for $a+2\le s\le2a$, every representation of $3b+2s$
has exactly three summands: the cases $q=1$ and $q\ge5$ are excluded
exactly as in Lemma~\ref{lem:minus-five}.  Hence the representations
are in bijection with the triples
$$
a+1\ge t_1\ge t_2\ge t_3\ge0,
\qquad t_1+t_2+t_3=s.
$$
Their number is $\brp{3}{a+1}{s}$.
\end{proof}

\begin{lemma}
\label{lem:plus-lower}
Let $a\ge3$.  Then
$$
r_+(5b+2s)>
\left\lfloor\frac{a+1}{2}\right\rfloor
\qquad(0\le s\le a+1)
$$
and
$$
\brp{3}{a+1}{s}>
\left\lfloor\frac{a+1}{2}\right\rfloor
\qquad(a+2\le s\le2a).
$$
\end{lemma}

\begin{proof}
For $s=0$, Lemma~\ref{lem:plus-layers} gives
$$
r_+(5b)=p_{\le3}(a+2)
\ge p_{\le2}(a+2)
=\left\lfloor\frac{a+2}{2}\right\rfloor+1
>
\left\lfloor\frac{a+1}{2}\right\rfloor.
$$
For $1\le s\le a+1$, the same lemma and the two-part lower bounds give
\begin{align*}
r_+(5b+2s)
&=p_{\le5}(s)+p_{\le3}(a+2-s)\\
&\ge
\left\lfloor\frac s2\right\rfloor+1
+
\left\lfloor\frac{a+2-s}{2}\right\rfloor+1\\
&>
\left\lfloor\frac{a+1}{2}\right\rfloor.
\end{align*}

For the final assertion, put $L=a+1$.  The interval
$[a+2,2a]$ is contained in $[L,2L]$.  Hence
Corollary~\ref{cor:three-part-bound} gives
$$
\brp{3}{a+1}{s}
\ge p_{\le3}(a+1)
\ge p_{\le2}(a+1)
=\left\lfloor\frac{a+1}{2}\right\rfloor+1.
$$
\end{proof}

\begin{proof}[Proof of Theorem~\ref{thm:plus} for $p\ge2$]
Let $j\ge2$ satisfy
$$
M_j\le p<M_{j+1}.
$$
Because
$$
M_j\le p\le\left\lfloor\frac{a+1}{2}\right\rfloor,
$$
Lemma~\ref{lem:index-bound} gives $j\le a$.  In particular,
Lemma~\ref{lem:plus-four} applies with $s=j$, and
$$
r_+(4b+2j)=p_{\le4}(j)=M_j\le p.
$$
Thus
\begin{equation}
4b+2j\notin S_p(A_a^+). \label{eq:plus-notin}
\end{equation}

We now bound every $p$-Ap\'ery element explicitly.  Since $b$ is
odd, the classes
$$
2s\pmod b,\qquad 0\le s\le2a,
$$
form a complete residue system modulo $b$.  We divide this complete
range into three cases.

\medskip
\noindent\textbf{Case 1: $0\le s\le j$.}
Lemma~\ref{lem:plus-lower} gives
$$
r_+(5b+2s)>
\left\lfloor\frac{a+1}{2}\right\rfloor
\ge p.
$$
Therefore
$$
5b+2s\in S_p(A_a^+),
\qquad
5b+2s\le5b+2j.
$$

\medskip
\noindent\textbf{Case 2: $j+1\le s\le a+1$.}
Since $j\ge2$, Lemma~\ref{lem:plus-four} applies and gives
$$
r_+(4b+2s)
=p_{\le4}(s)
\ge p_{\le4}(j+1)
=M_{j+1}>p.
$$
Hence $4b+2s\in S_p(A_a^+)$.  Furthermore,
$$
4b+2s
\le4b+2(a+1)
=10a+6
<10a+5+2j
=5b+2j.
$$

\medskip
\noindent\textbf{Case 3: $a+2\le s\le2a$.}
By Lemmas~\ref{lem:plus-layers} and~\ref{lem:plus-lower},
$$
r_+(3b+2s)
=\brp{3}{a+1}{s}
>
\left\lfloor\frac{a+1}{2}\right\rfloor
\ge p.
$$
Thus $3b+2s\in S_p(A_a^+)$, and
$$
3b+2s
\le3b+4a
=10a+3
<5b+2j.
$$

The three cases cover every residue class modulo $b$, and in each
class we have exhibited a member of $S_p(A_a^+)$ not exceeding
$5b+2j$.  Therefore
\begin{equation}
\max\Ap_p(A_a^+;b)\le5b+2j. \label{eq:plus-upper}
\end{equation}
Case~1 with $s=j$ shows that $5b+2j\in S_p(A_a^+)$.

It remains to prove that this is the $p$-Ap\'ery element in the class
$2j\pmod b$.  Since $j\le a$, one has $0\le2j<b$.  If a smaller
member $kb+2j$, with $0\leq k<5$, belonged to $S_p(A_a^+)$, then
Lemma~\ref{lem:closure} would imply
$$
4b+2j=(kb+2j)+(4-k)b\in S_p(A_a^+),
$$
contrary to \eqref{eq:plus-notin}.  Consequently,
$$
m_{2j}^{(p)}=5b+2j.
$$
Combining this equality with \eqref{eq:plus-upper}, we obtain
$$
\max\Ap_p(A_a^+;b)=5b+2j.
$$
Finally,
$$
g_p(A_a^+)
=(5b+2j)-b
=4b+2j
=8a+2j+4
$$
by Lemma~\ref{lem:apery-formula}.
\end{proof}
\section*{Acknowledgements}
This research was supported by the ANCHOR program Glocal through the Jeju ANCHOR center, funded by the Ministry of Education(MOE) and the Jeju Special Self-Governing Province, Republic of Korea.(2026-ANCHOR-17-001)

\end{document}